\documentclass{amsart}
\usepackage[latin1]{inputenc}
\usepackage{amsmath,amsfonts,amssymb}
\usepackage{amscd,graphicx}

\newcommand{\la}{\lambda}

\newcommand{\N}{{\mathbb{N}}}
\newcommand{\ds}{\displaystyle }

\newtheorem{theorem}{Theorem}[section]

\newtheorem{prop}[theorem]{Proposition}
\newtheorem{cor}[theorem]{Corollary}
\newtheorem{lemma}[theorem]{Lemma}
\newtheorem{rem}[theorem]{Remark}

\begin{document}

\title[Unique Path Partitions]
{Unique path partitions:  \\ Characterization and Congruences}

\author{Christine Bessenrodt}

\address{Institut f\"ur Algebra, Zahlentheorie und Diskrete Mathematik
\\
Leibniz Universit\"at Hannover, Welfengarten 1
\\
D-30167 Hannover, Germany }
\email{bessen@math.uni-hannover.de}

\author{J\o rn B. Olsson}
\address{
Department of Mathematical Sciences, University of Copenhagen\\
Universitetsparken 5,DK-2100 Copenhagen \O, Denmark
}
\email{olsson@math.ku.dk}

\author{James A. Sellers}
\address{Department of Mathematics, Penn State University \\
104 McAllister Building, University Park, PA 16802, USA
}
\email{sellersj@math.psu.edu}

\date{September 3, 2012}


\begin{abstract}
We give a complete classification of the unique path partitions and
study congruence properties of the function which enumerates such partitions.
\end{abstract}

\maketitle

\medskip

\section{Introduction}

\medskip

The famous Murnaghan-Nakayama formula gives a combinatorial
rule for computing the value of the irreducible character  of the
symmetric groups $S_n$ labelled by
the partition $\la$ on the conjugacy class labelled by a partition $\mu$ (see \cite{JK}).
This value is the weighted sum over the $\mu$-paths in $\la$, as
defined below, where
the weight is a sign corresponding to the sum of the leg lengths
of the rim hooks removed along the path.

If  $\mu=(a_1,a_2,...,a_k)$, with $a_1\geq a_2\geq \ldots \geq a_k >0$,
and $\la$ are  partitions of~$n$,
 then a $\mu$-path in $\la$ is a sequence of partitions, $\la=\la_0,
 \la_1,..., \la_k=(0)$,   where for $i=1,...,k$ the partition $\la_i$ is
obtained by removing an $a_i$-hook from $\la_{i-1}$.
As in \cite{OlSC},  we call $\mu$ a {\it unique path partition for $\la$} 
(or {\it up}-partition for $\la$ for short) if the number of
$\mu$-paths in $\la$ is at most~1.
We call $\mu$ a {\it up}-partition if it is a {\it up}-partition
for all partitions of~$n$.

Thus, a {\it up}-partition $\mu$ labels a conjugacy class where
all non-zero irreducible character values are $1$ or $-1$,
i.e., they are {\it sign partitions} as defined in \cite{OlSC}.
By \cite[7.17.4]{EC2}, the sign partitions $\mu$ are exactly those for which
the expansion of the corresponding power sum symmetric function
into Schur functions is
multiplicity-free.

Note that not every sign partition is a {\it up}-partition as cancellation may occur.
For example, the partition $(3,2,1)$ is a sign partition, but not a
{\it up}-partition, since there are two  $(3,2,1)$-paths in
the partition $(3,2,1)$.

In this paper, we accomplish three goals.  First, we provide
an explicit classification of the unique path partitions
in terms of partitions we call {\it strongly decreasing}.  We then discuss numerous connections between {\it up}-parti\-tions and certain types of binary partitions.  Such connections are truly beneficial; they led us to the development of a generating function for, and a recurrence satisfied by $u(n)$, the number of {\it up}-partitions of the positive integer $n$.  Thanks to this link between {\it up}-partitions and restricted binary partitions, we were encouraged to consider the arithmetic properties of $u(n)$.  (Such a motivation is natural based on the literature that already exists on congruence properties satisfied by binary partitions.  Indeed, Churchhouse \cite{CH} initiated the study of congruence properties satisfied by the unrestricted binary partition function in the late 1960's.  This work was further extended by R\o dseth and Sellers~\cite{RS}.)
We close this paper by proving a number of congruence relations satisfied by $u(n)$ modulo powers of~2.

\section{The classification of {\it up}-partitions}

We now collect the facts necessary for classifying
the {\it up}-partitions in an elegant fashion.
As usual, we gather equal parts together and write $i^m$ for $m$ parts
equal to~$i$ in a partition.

\begin{lemma}\label{lem:up}
(1) If $\mu=(a_1,a_2,...,a_k)$ is a {\it up}-partition
with $a_k=2$,  then
$\mu'=(a_1,a_2,...,a_{k-1}, 1^2)$ is also a {\it up}-partition.

\smallskip

(2) If $\mu=(a_1,a_2,...,a_k)$ is a {\it up}-partition
with $k \ge 2$,  then
$\mu_2=(a_2,...,a_k)$ is also a {\it up}-partition.
\end{lemma}

\proof (1) follows immediately from the definition.

(2) If a partition $\la_2$ of $n-a_1$ has two or more
$\mu_2$-paths then any partition of $n$ obtained 
by adding an $a_1$-hook
to $\la_2$ has two or more $\mu$-paths.
\qed

\medskip

\begin{lemma}\label{upext}
Let $\mu=(a_1,a_2,...,a_k)$ be a
  partition of $n$ and $a>n$.
Then $\mu$ is a {\it up}-partition if and only if
$\mu'=(a,a_1,...,a_k)$ is a {\it up}-partition.
\end{lemma}

\proof
By Lemma \ref{lem:up}(2)  we only need to show
that if $\mu$ is a {\it up}-partition then also $\mu'$ is
a {\it up}-partition. Let $\la'$ be a partition of $a+n$. Since $a>n$,
$\la'$ cannot contain two or more $a$-hooks. If $\la'$ contains an
$a$-hook,  we let $\la$ be the partition obtained by removing
it. Since by assumption $\mu$ is a {\it up}-partition for  $\la$,  we get
that $\mu'$ is a {\it up}-partition for~$\la'$.
\qed

\medskip

We call an extension of a partition of $n$ by a part $a>n$ as in Lemma~\ref{upext}
{\it strongly decreasing}, or for short, an {\it sd-extension}.
A partition $\mu$ obtained from a partition $\rho$ by several
{\it sd}-extensions is then called an {\it sd-extension of $\rho$}; if $\rho=(0)$, $\mu$ is
called an  {\it sd-partition}.
As stated in \cite{OlSC}, a partition  $\mu=(a_1,a_2,...,a_k)$ is an {\it sd}-partition if  and only if
$a_i>a_{i+1}+...+a_k$ for all $i=1,...,k-1$.

\smallskip

We have the following classification result for {\it up}-partitions:

\begin{theorem}\label{upclass}
A partition  $\mu$ is a {\it up}-partition if and only if one of the following
holds:

{\rm (i)} $\mu$ is an {\it sd}-partition.

{\rm (ii)}  $\mu$ is an {\it sd}-extension of $(1^2)$.
\end{theorem}

\proof
In the proof we use the well-known connection between first column hook lengths
and hook removal as described in \cite[Section 2.7]{JK}.

As $(0)$ and $(1^2)$ are {\it up}-partitions, Lemma~\ref{upext}
shows that their {\it sd}-extensions are {\it up}-partitions.
Suppose that $n$ is minimal such that there exists a partition
$\mu=(a_1,a_2,...,a_k)$ of $n$,  which is a {\it up}-partition but not an
{\it sd}-extension of $(0)$ or $(1^2)$.
Obviously $k \ge 2$.

Assume $a_2=1$, i.e., $\mu=(n-k+1,1^{k-1})$.
If $k>3$, then $\mu$ is not a {\it up}-partition since $(1^{k-1})$ is not.
For $k=3$, only $(2,1^2)$ and $(1^3)$ are not {\it sd}-extensions of $(1^2)$,
but these are not {\it up}-partitions.
For $k=2$, $\mu$ is an {\it sd}-partition or $(1^2)$.

Thus we may now assume that $a_2>1$.
We put $\mu_i=(a_i,...,a_k)$ and $n_i=|\mu_i|$ for $i=2,...,k$. Also $n_{k+1}:=0$.

Now suppose that $a_1=a_2$.
If $k=2$ then $\mu$ is not a {\it up}-partition for $\la=(a_1,a_1)$.
If $k>2$ then $\mu$ is not a {\it up}-partition for $\la=(n-a_1,1^{a_1})$.

Thus we may now assume $a_1>a_2>1$.
By Lemma~\ref{lem:up},  $\mu_2=(a_2,...,a_k)$ is a {\it up}-partition,
and thus, by minimality, it is an {\it sd}-extension of $(0)$ or $(1^2)$.
Then $\mu$ cannot be an {\it sd}-extension of $\mu_2$, and hence
$a_1\le n_2$.

Now  $a_1>a_2> n_3$ and hence  $d:=a_1-n_3-1>0$.
Note that $n_2=n_3+a_2>n_3+1$, and thus $\la=(n_2,n_3+1,1^d)$ is a partition of
$n_2+n_3+1+d=a_1+n_2=n$.
The set of first column hook lengths for $\la$ is $\{a_1+a_2,a_1,d,d-1,...,1\}, $
as is easily calculated.
As $d\le n_2-n_3-1=a_2-1$,  $\la$ has two $a_1$-hooks. After
removing the $a_1$-hook in the second row we get the partition
$\la'=(n_2)$.
After removing the $a_1$-hook in the first row we get $\{a_1,a_2, d,d-1,...,1\} $
as a set of a first column hook lengths
for a partition $\la''$.
Now $\la''$ has an $a_2$-hook in
the second row. Removing it we obtain the partition $(n_3)$.
This shows that $\mu$ is not a {\it up}-partition for $\la$, giving a contradiction.
\qed

\smallskip

\section{On {\it up}-partitions and restricted binary partitions}

For each $n\in \N$, we denote the number of {\it up}-partitions of $n$ by $u(n)$.
For $t\in \mathbb{N}$, we define an {\it $sd_t$-partition} to be an $sd$-extension of the partition~$(t)$.
The following lemma is obvious.

\begin{lemma}\label{sdt} Let $\mu$ be a partition of $t$. There is a bijection
  between $sd$-extensions of $\mu$ and $sd_t$-partitions obtained by replacing
  all the parts of $\mu$ by one part~$t$.
\qed
\end{lemma}

We denote {\it the number of sd-partitions} of $n$ by $s(n)$ and {\it
  the number of $sd_t$-partitions} of $n$ by $s_t(n)$ so that
$s(n)=\sum_{t \ge 1}s_t(n)$.
Combining  Theorem~\ref{upclass} with Lemma~\ref{sdt} we get the following:

\begin{cor}\label{cor:u-as-s}
For each $n\geq 1$,
$$ u(n)=s(n)+s_2(n)\:.  \qed $$
\end{cor}

Next, we focus our attention on $s(n)$.

\begin{prop}\label{prop:s-decomp}
For each $n \ge 2$,
$$s(n)=2s_1(n)+s_2(n)\:.$$
\end{prop}

\proof
Let $\la=(a_1,a_2,...,a_k)$ be an $sd_t$-partition, i.e., $a_k=t$.
If we  map $\la$ onto $(a_1,a_2,...,a_k-1,1)$ we get a bijection
between the set of all $sd_t$-partitions of $n$ with $t \ge 3$ and the set of all
$sd_1$-partitions of $n$.  Thus $s_1(n)=\sum_{t \ge 3}s_t(n)$.  The
result follows, since $s(n)=\sum_{t \ge 1}s_t(n)$.
\qed

\medskip

Combining Corollary \ref{cor:u-as-s} and Proposition \ref{prop:s-decomp}, we have
the following:
\begin{theorem}\label{u-pari}
For each $n \ge 2$, $u(n)$ is even. In fact,
$$\frac{u(n)}{2}=s_1(n)+s_2(n)\:. \qed $$
\end{theorem}
Thanks to their definition, it is clear that $sd$-partitions are closely related to
non-squashing partitions and binary partitions as described in~\cite{SlSe}.
A partition
$\la=(a_1,a_2,...,a_k)$ is called
{\it non-squashing} if
$a_i \ge a_{i+1}+...+a_k$ for $1 \le i \le k-1$
and {\it binary} if all parts $a_i$ are powers of $2$.
The difference between $sd$- and non-squashing partitions is whether or not the
inequalities between $a_i$ and $a_{i+1}+...+a_k$ are strict.
A binary partition is called {\it restricted} (for short,
an {\it $rb$-partition}) if it satisfies
the following condition: Whenever $2^i$ is a part and $i \ge 1$ then
$2^{i-1}$ is also a part.
For $t \in \N,$ an $rb_t$-{\it partition} is an $rb$-partition
where the largest part occurs with multiplicity~$t$.

With this in mind, we can naturally connect the $sd_t$-partitions and the $rb_t$-partitions.

\begin{theorem}\label{sdt-bij}
Let $n,t \in \N$.  There is a bijection between the set of
$sd_t$-partitions of $n$ and the set of $rb_t$-partitions of~$n$.
\end{theorem}
\proof
Clearly, an $sd$-partition $\la=(a_1,a_2,...,a_k)$ of $n$ is uniquely
determined  by the positive integers $d_i \in \N$, $i=1,...,k$, defined
by $d_i=a_i-(a_{i+1}+...+a_k)$ for $i=1,...,k-1$, and  $d_k=a_k$.
An easy calculation shows that with this notation
$n=d_1+d_22+...+d_k2^{k-1}$.  Thus if we map $\la$ onto the
binary partition where $2^j$ occurs with multiplicity $d_{j+1}$,  $j=0,1,...,k-1$,
we get the desired bijection.
\qed

\begin{rem}\label{Rem3.6} 
{\rm
Theorem \ref{sdt-bij} shows that $s(n)$ equals
the number of $rb$-partitions of~$n$.   Let
$S(q):=\sum_{n\ge 1}s(n)q^n$ be the generating function for $s(n)$.
It is easy to write down the
generating function for the number of $rb$-partitions which implies that
$$S(q)=\sum_{i \ge 1} q^{2^i-1}\prod_{j=0}^{i-1}\frac{1}{1-q^{2^j}}\:.$$ 
From its definition, one also gets the identity 
$$S(q)(1-q)=q(1+S(q^2))\:.$$ 
Moreover, the generating function $S_t(q)$ for the number of $rb_t$-partitions is given by
$$S_t(q)=\sum_{i \ge 1}
q^{2^i-1+(t-1)2^{i-1}}\prod_{j=0}^{i-2}\frac{1}{1-q^{2^j}},$$ 
and it satisfies the identity 
$$(S_t(q)-q^t)(1-q)=qS_t(q^2)\:.$$
Hence, by Theorem \ref{u-pari}, 
the generating function $U(q)$ for the number of {\it up}-partitions
is then
$$U(q)=2(S_1(q)+S_2(q)).$$
}
\end{rem}

We now exploit this connection between $rb$-partitions and
$sd$-partitions to prove a number
of facts about $s(n)$ and related functions. 
The following results may alternatively also be proved by using 
the identities for the generating functions $S(q)$ and $S_t(q)$ 
stated above. 

\begin{prop}\label{s-pari}
For each $r\in \N$ we have
\begin{center}
$
\begin{array}{rcl}
s(2r)&=&s(2r-1)  \\
s(2r+1)&=&s(2r)+s(r) \;.
\end{array}
$
\end{center}
\end{prop}

\proof
An $rb$-partition must contain a part~1. Removing such a part from an
$rb$-partition $\la$ of $2r$
gives an $rb$-partition $\la'$ of $2r-1$. (A binary partition of an odd
number must contain 1 as a part, so that $\la'$ is still $rb$.) This map
is then in fact a bijection between $rb$-partitions of $2r$ and those of~$2r-1$.

Removing a part 1 from an $rb$-partition $\la$ of $2r+1$ gives a
binary partition $\la'$ of $2r$.  If $\la'$ has a part equal to~1, it is
an $rb$-partition and we put $\la''=\la'$.  Otherwise all parts of
$\la'$ are even and we may divide them all by~2 to get an
$rb$-partition $\la''$ of~$r$.  The process of going from $\la$ to $\la''$
may obviously be reversed. Thus $s(2r+1)=s(2r) +s(r).$
\qed

\smallskip

With Proposition \ref{s-pari} in mind, we define $s^*(r):=s(2r) (=s(2r-1)) $
for $r \in \N$.

\begin{prop}\label{s-recur}
We have $s^*(1)=1$ and
$$s^*(r)=s^*(r-1)+s^*(\left\lfloor \frac{r}{2} \right\rfloor)\;,
\text{ for $r \ge 2$.}$$
\end{prop}
\proof
Clearly $s^*(1)=s(1)=1$.
We prove the proposition by showing that for $r' \in \N$ we have
$$s^*(2r')=s^*(2r'-1)+s^*(r') \: \text{ and } \: s^*(2r'+1)=s^*(2r')+s^*(r')\:.$$
The equations are by definition of $s^*$ equivalent to
$$s(4r')=s(4r'-2)+s(2r') \: \text{ and }  \: s(4r'+2)=s(4r')+s(2r')\:.$$
But these are easily deduced from Proposition~\ref{s-pari}.
\qed

\medskip

\begin{rem}
{\rm
Proposition \ref{s-recur} proves that the sequence $s^*(n)$ is listed in
\cite{OEIS} as A033485 and thus that the sequence $s(n)$ is listed
as A040039.
In particular, the comment by John McKay which appears in A40039 in \cite{OEIS} is confirmed.
}
\end{rem}

We proceed to consider the numbers $s_t(r)$ of $rb_t$-partitions.
\begin{prop}\label{st-pari}
Let $t \in \N$.
We have $s_t(1)=s_t(2)=...=s_t(t-1)=0$, $s_t(t)=1$,
$s_t(t+1)=...=s_t(2t)=0$,  and $s_t(2t+1)=1$.
Also,
$s_t(2r)=s_t(2r-1)$ whenever $t \neq 2r, 2r-1$,
i.e., whenever  $r \ne \lfloor \frac{t+1}{2}\rfloor$.
\end{prop}
\proof
The statements about $s_t(j)$ for $j \leq 2t+1$ are trivial.
The final statement is proved in analogy with Proposition \ref{s-pari}.
Using the notation of that proof we have the following:
If we assume that $\la$ is $rb_t$ then also $\la'$ is $rb_t$ with the
exception of the case where $\la=(1^t)$.  Also, if $\la'$ is $rb_t$
then $\la$ is $rb_t$ with the exception of the case where $\la'=(1^t)$.
Thus we have
$s_t(2r)=s_t(2r-1)$ except when $t \in \{2r,2r-1\}$ .
\qed

\medskip

\begin{cor}\label{cor:u-evenodd}
We have $u(1)=1, u(2)=2$, and
for $r \geq 2,  u(2r)=u(2r-1)$.
\qed 
\end{cor}

We now define
$$s^*_t(r):=
\left\{
\begin{array}{cl}
s_t(2r) &  \text{if $r$ is odd} \\
s_t(2r-1)  & \text{if $r$ is even}
\end{array}
\right.\:.$$
Proposition \ref{st-pari} shows
that for all $r \ne \lfloor \frac{t+1}{2}\rfloor$ we
have  $s^*_t(r)=s_t(2r-1)=s_t(2r)$.
Also  $s^*_t(r)=0$ for $1 \le r \le t$,
$r \ne\lfloor \frac{t+1}{2}\rfloor$ and  $s^*_t(t+1)=1$.

\smallskip

In analogy with Proposition~\ref{s-recur} we have:
\begin{prop}\label{s*recur}
For all $r \ge t+2$,
$s^*_t(r)=s^*_t(r-1)+s^*_t(\lfloor \frac{r}{2} \rfloor)$.
\end{prop}
\proof
The assumption on $r$ and $t$ implies that the partitions $(1^r)$ and
$(1^{r-1})$ are not~$rb_t$. Therefore the bijections in the proof of
Proposition \ref{s-pari} work for $rb_t$-partitions of $r$ as well.
Thus we have the recursions
\begin{center}
$
\begin{array}{rcl}
s_t(2r)&=&s_t(2r-1) \\
s_t(2r+1)&=&s_t(2r)+s_t(r)\:.
\end{array}
$
\end{center}
In the case $t=1, r=3$ we have $s^*_1(3)=s_1(6)=1$ and
$s^*_1(2)+s^*_1(1)=s_1(3)+s_1(2)=1+0=1$.
We assume $r \ge 4$ and write  $r=4s+i$, $s\in \N$,
$i \in \{0,1,2,3\}$. We show
\begin{center}
$
\begin{array}{l}
s^*_t(4s)=s^*_t(4s-1)+s^*_t(2s) \\
s^*_t(4s+1)=s^*_t(4s)+s^*_t(2s)\\
s^*_t(4s+2)=s^*_t(4s+1)+s^*_t(2s+1)\\
$$s^*_t(4s+3)=s^*_t(4s+2)+s^*_t(2s+1)
\end{array}
$
\end{center}
By definition of $s^*_t$ the equations are equivalent to
\begin{center}
$
\begin{array}{l}
s_t(8s-1)=s_t(8s-2)+s_t(4s-1)\\
s_t(8s+2)=s_t(8s-1)+s_t(4s-1)\\
s_t(8s+3)=s_t(8s+2)+s_t(4s+2)\\
s_t(8s+6)=s_t(8s+4)+s_t(4s+2)\:.
\end{array}
$
\end{center}
These follow from the recursions for $s_t$.
\qed

\medskip

Lastly, we define \ $w(n)=\frac{u(2n)}{2}$. Theorem~\ref{u-pari}
and Proposition~\ref{s*recur} yield the following:
\begin{prop} \label{u*1recur}
For each $n\geq 1$,  $w(n)=s^*_1(n)+s^*_2(n)$.
Moreover,  for $n \ge 3$,
$w(n)=w(n-1)+w(\lfloor \frac{n}{2} \rfloor)$,
and $w(1)=w(2)=1$.
\qed
\end{prop}

\begin{rem}
{\rm
Proposition \ref{u*1recur} shows that the sequence
of numbers $w(n)$ is listed in
\cite{OEIS} as A075535. The simple recurrence relation is  used in the
next section to prove congruence results for the numbers
$w(n)$ and thus for the
numbers $u(n)$ of unique path partitions.
}
\end{rem}

\begin{rem}\label{rem:w_i-seq}
{\rm
We may consider also
$w_2(n):= s^*_3(n)+s^*_4(n)$.  Then we have
$w_2(1)=0, w_2(2)=1, w_2(3)=0, w_2(4)=1$, and  for $n \ge 5$ \\
\centerline{$w_2(n)=w_2(n-1)+w_2(\lfloor\frac{n}{2}\rfloor)$.}
Similar recurrence relations are more generally valid for
\\
\centerline{$w_r(n):=s^*_{2r-1}(n)+s^*_{2r}(n)$} 
which starts by
$w_r(1)=...=w_r(r-1)=0, w_r(r)=1, w_r(r+1)=...=w_r(2r-1)=0, w_r(2r)=1$.
This is an infinite family of sequences which may satisfy
congruence relations similar to those satisfied by $w_1(n)=w(n)$.
}
\end{rem}

In the next section we discuss congruences for
$w(n)$ and in part also for the $w_i(n)$'s.

\section{Congruences for the number of {\it up}-partitions}

In this section we investigate arithmetical
properties of $u(n)$,  the number
of {\it up}-partitions of~$n$.
Since $w(n)=\frac{u(2n)}2$, any result on the $w$-sequences
may be translated into a result on the $u$-sequence.
In particular, as 
studying congruences of the $u$-sequence modulo~$2m$
is equivalent to studying the $w$-sequence modulo~$m$,
we will concentrate on the latter sequence.

\smallskip

At the start, we consider a more general situation that
also covers the more general sequences defined in Remark~\ref{rem:w_i-seq};
however, in the remaining part of this section we restrict our attention
to the numbers~$w(n)$.

\begin{prop}\label{prop:oddeven-her}
Let $(a(n))_{n\in \N}$ be a sequence with\, $a(c)$, $a(2c)$ odd
for some $c\in \N$, $a(m)$ even when $c < m < 2c$,
and $a(n)=a(n-1)+a(\lfloor \frac n2 \rfloor )$ for $n\geq 2c$.
Then for $n\geq c$, $a(n)$ is odd exactly when $n$ is of the form
$2^d c$.
\end{prop}

\proof
Certainly the assertion is true for $n=c$ and~$n=2c$.
Assume the result holds up to some number $n=2^r c$, $r\geq 1$.
Then
$$a(n+1)=a(n)+a(\lfloor \frac n2 \rfloor )= a(2^r c)+a(2^{r-1} c)
\equiv 0 \mod 2 \:.$$
For any $k$ with $2\leq k \leq 2^r c -1$,  we then get by induction on $k$ that
$$a(n+k)=a(n+k-1)+a(\lfloor \frac{n+k}2 \rfloor ) \equiv 0 \mod 2\:$$
since $2^{r-1} c<\lfloor \frac{n+k}2 \rfloor < 2^r c$.
For $k=2^r c$ we then obtain
$$a(2^{r+1}c)=a(2^{r+1}c-1)+a(2^r c) \equiv 1 \mod 2\:.$$
Hence the assertion is proved.
\qed

\medskip

\begin{cor}
Let $(a(n))_{n\in \N}$  be as in Proposition~\ref{prop:oddeven-her}.
Let $m$ be an odd number such that
$2^b c+1<m\leq 2^{b+1}c-1$ for some $b$.
Then $a(m) \equiv a(m-2) \mod 4$.
In particular, $a(m)\equiv a(2^bc+1) \mod 4$.
\end{cor}

\proof
Since $m$ is odd, we have
$$a(m)=a(m-1)+a(\lfloor \frac m2 \rfloor )
= a(m-2)+2\ a(\lfloor \frac m2 \rfloor )\:.$$
As $m-1$ is not of the form $2^dc$, $\lfloor \frac m2 \rfloor > 2^{b-1}c$ is not
either. Hence, $a(\lfloor \frac m2 \rfloor )$ is even,
and then the claim follows.
\qed

\medskip

Since $w(1)=w(2)=1$, the following is immediate,
and it gives corresponding congruences modulo~4 and~8 for $u(n)$:

\begin{cor}\label{cor:w-basic}
For $n\geq 1$,
$w(n)$ is even exactly when $n$ is not a 2-power.

For any odd number $m$ such that $2^b +1 \leq m \leq 2^{b+1}-1$,
\\
\centerline{$w(m) \equiv w(2^b+1) \mod 4$.}
\qed
\end{cor}

\medskip

Note that the first part of Corollary \ref{cor:w-basic} implies infinitely many Ramanujan--like
congruences modulo 4 satisfied by $u(n)$.
To further understand the congruences of $u(n)$ mod~8, we first focus on the 2-powers.
Set 
$v(k)=w(2^k)$ for $k\in \N_0$.

\begin{prop}\label{prop:v-rec}
For each $k\geq 2$,
$$v(k) \equiv 2 v(k-1)+v(k-2) \mod 4\:.$$
\end{prop}
\proof
Using Corollary \ref{cor:w-basic}, we have the following congruences mod~4:
$$
\begin{array}{rcl}
v(k)&=&w(2^k)=w(2^{k-1})+w(2^k-1)
\equiv w(2^{k-1})+w(2^{k-1}+1) \\[6pt]
&\equiv & 2\ w(2^{k-1}) + w(2^{k-2})= 2\ v(k-1) + v(k-2)\:.
\end{array}
$$
\vspace{-6ex}

\qed

\medskip

\begin{prop}\label{prop:v-mod8}
For each $k\geq 1$,
$$
v(k)=w(2^k) \equiv
\left\{
\begin{array}{cll}
k & \mod 8 & \text{if $k$ is odd} \\
k+1 & \mod 8 & \text{if $k$ is even}
\end{array}
\right.\:.
$$
Equivalently,
$$v(k) \equiv 2 \left\lfloor \frac{k}2 \right\rfloor +1 \mod 8 \:.$$
\end{prop}

\proof
From the recursion formula we have
$$
\begin{array}{rcl}
w(2^k)
&=& w(2^{k-1})+w(2^k-1) = w(2^{k-1})+w(2^{k-1}-1)+w(2^k-2)\\[8pt]
&=& w(2^{k-1})+2 w(2^{k-1}-1)+w(2^k-3) \\
&\vdots & \\
&=& w(2^{k-1})+ 2 w(2^{k-1}-1)+...+ 2 w(2^{k-2}+1)+w(2^{k-1}+1)\\[8pt]
&=&  2 w(2^{k-1})+ 2 w(2^{k-1}-1)+...+ 2 w(2^{k-2}+1)+w(2^{k-2})
\end{array}
$$
and we now investigate sums of the form
$\ds \sum_{i=2^d+1}^{2^{d+1}} w(i)$,
for $d\geq 1$.
We want to show by induction that they are
always congruent to $5$~mod~$8$; for $d=1$, $w(3)+w(4)=2+3=5$, so the
claim holds.
Now we have for any $d\geq 2$ (using induction and the corollary):
$$
\begin{array}{rcl}
\ds
\sum_{i=2^d+1}^{2^{d+1}} w(i)
&=&
\ds
\sum_{i=2^{d-1}+1}^{2^{d}} w(2i)+\sum_{i=2^{d-1}+1}^{2^{d}} w(2i-1)
\\[8pt]
&=&
\ds
\sum_{i=2^{d-1}+1}^{2^{d}} w(i) + 2 \sum_{i=2^{d-1}+1}^{2^{d}} w(2i-1)
\\[12pt]
&\equiv & 5 + 2^d w(2^{d}+1) \mod 8 \\[8pt]
& \equiv & 5 \mod 8 \:.
\end{array}
$$
We can now continue to compute $w(2^k) \mod 8$ for $k\geq 2$:
$$
\begin{array}{rcl}
w(2^k)
&= &
\ds 2 \sum_{i=2^{k-2}+1}^{2^{k-1}} w(i) + w(2^{k-2}) \\[8pt]
&\equiv & 2 + w(2^{k-2}) \mod 8 \:.
\end{array}
$$
Starting with $w(2^0)=1=w(2^1)$, the assertion now follows easily.
\qed

\medskip

We now obtain full information on the congruences modulo~8 for
the $u$-sequence via the following result on the $w$-sequence
modulo~4.
\begin{theorem}\label{thm:w-mod4}
Let $n\in \N$, $n$ not a 2-power.
Write $\ds n=\sum_{i=0}^k 2^{n_i}$ with $n_0<n_1<\ldots <n_k$.
Then we have
$$
w(n) \equiv
\left\{
\begin{array}{cl}
0 \mod 4 & \text{ if } n_0 \equiv 3 \mod 4 \\
& \text{ or } n_0 \equiv 0 \mod 4 \text{ and $n_k$ is even}\\
& \text{ or } n_0 \equiv 2 \mod 4 \text{ and $n_k$ is odd}\\
2 \mod 4 & \text{if } n_0 \equiv 1 \mod 4 \\
& \text{ or } n_0 \equiv 0 \mod 4 \text{ and $n_k$ is odd}\\
& \text{ or } n_0 \equiv 2 \mod 4 \text{ and $n_k$ is even}\\
\end{array}
\right.
$$
\end{theorem}

\begin{proof}
Assume that $n_0\geq 1$; then $m=n-1$ is an odd number such that
$2^{n_k}+1\leq m=n-1 \leq 2^{n_k+1}-1$;
hence, using Corollary~\ref{cor:w-basic},
$w(n-1)\equiv w(2^{n_k}+1)=w(2^{n_k})+w(2^{n_k-1}) \mod 4$.
Then
$$
w(n) = w(n-1)+w(\sum_{i=0}^k 2^{n_i-1})
\equiv w(2^{n_k})+w(2^{n_k-1})+w(\sum_{i=0}^k 2^{n_i-1}) \mod 4\:.
$$
If $n_0>1$, we can repeat the argument to
obtain (using Corollary~\ref{cor:w-basic} again)
$$
\begin{array}{rcl}
w(n)& = & \ds w(n-1)+w(\sum_{i=0}^k 2^{n_i-1}) \\[6pt]
&\equiv & \ds v({n_k})+2\ v({n_k-1})+v({n_k-2})+w(\sum_{i=0}^k 2^{n_i-2})\mod 4\\[6pt]
&\equiv & \ds 2\ v(n_k)+w(\sum_{i=0}^k 2^{n_i-2}) \mod 4
\;\; \text{ (using Proposition \ref{prop:v-rec})}\\[6pt]
&\equiv & \ds 2 + w(\sum_{i=0}^k 2^{n_i-2}) \mod 4 \:.
\end{array}
$$

We now use this reduction to discuss the different cases for $n_0$.

If $n_0=4j-1$ for some $j\in \N$, then we can use the 2-step
reduction above $2j-1$ times, then the 1-step reduction, and we obtain
(using Corollary~\ref{cor:w-basic}~again)
$$
\begin{array}{rcl}
w(n)&\equiv & \ds 2 + w(2+\sum_{i=1}^k 2^{n_i-n_0+1}) \mod 4\\[6pt]
&\equiv & \ds 2+ w(2^{n_k-n_0+1})+w(2^{n_k-n_0})+w(1+\sum_{i=1}^k 2^{n_i-n_0}) \\[6pt]
&\equiv & 2+ w(2^{n_k-n_0+1})+w(2^{n_k-n_0})+w(1+2^{n_k-n_0}) \\[6pt]
&\equiv & 2+ w(2^{n_k-n_0+1})+2\ w(2^{n_k-n_0})+w(2^{n_k-n_0-1}) \\[6pt]
&\equiv & 2+ 2\ v(n_k-n_0+1) \equiv 0 \mod 4
\;\; \text{ (using Proposition \ref{prop:v-rec})}\:.
\end{array}
$$

In the case $n_0=4j+1$ for some $j\in \N$, we are just doing one less 2-step reduction,
hence in this case it follows that $w(n)\equiv 2 \mod 4$.

When $n_0=4j$ for some $j\in \N$,
we do again $2j-1$ 2-step reductions and obtain
$$
\begin{array}{rcl}
w(n)&\equiv & \ds 2 + w(2^2+\sum_{i=1}^k 2^{n_i-n_0+2}) \mod 4\\[6pt]
&\equiv & \ds 2+  w(3+\sum_{i=1}^k 2^{n_i-n_0+2})+ w(2+\sum_{i=1}^k 2^{n_i-n_0+1})\\[6pt]
&\equiv & 2+  w(2^{n_k-n_0+2}+1)+ 2\\[6pt]
&\equiv & w(2^{n_k-n_0+2})+w(2^{n_k-n_0+1}) \\[6pt]
&\equiv & v({n_k+2})+v({n_k+1}) \mod 4 \:.
\end{array}
$$
With the previous result on the $v$-sequence, the assertion then follows.

When $n_0=0$, we are in the case of an odd $n$, where then (by Corollary~\ref{cor:w-basic})
$$w(n)=w(1+2^{n_k})=w(2^{n_k})+w(2^{n_k-1})$$
and the result is the same as above for $n_0=4j$.

When $n_0=4j-2$ for some $j\in \N$, the result is complementary
to the one above, by a shift of 2, as stated in the assertion.
\end{proof}

\begin{rem}
{\rm 
In Section~3 we have seen that the generating function $W(q)$ of $w(n)$ 
is the even part of $S_1(q)+S_2(q)$. 
The functional equations given in Remark~\ref{Rem3.6} then yield 
$$W(q)=q + \frac{1+q}{1-q}W(q^2)\:.$$
Iterating this equation and considering congruences modulo $2$ and modulo $4$ 
then provides a different route to the congruence results obtained above. 
}
\end{rem}

We close by noting that there may also be very special 
behavior of the $w$-sequence modulo~8. 
(Indeed, the data strongly suggest this.)
Obviously, this would then imply congruences modulo 16 for the numbers~$u(n)$.

\bigskip

{\bf Acknowledgments.}
The first two authors would like to thank the Danish Research Council (FNU)
for the support of their collaboration. 
Thanks go also to a referee for remarks on alternative proofs.

\end{document}